\documentclass[12pt]{amsart}
\usepackage[active]{srcltx}

\def\ID{{\Bbb D}}
\def\IB{{\Bbb B}}

\newcommand*{\pd}[3][]{\ensuremath{\frac{\partial^{#1} #2}{\partial #3}}}

\usepackage[psamsfonts]{amssymb}
\usepackage{amsfonts,amsmath}
\usepackage{amsthm}
\usepackage{graphicx}

\newcounter{minutes}
\divide\time by 60
\newcounter{hours}
\multiply\time by 60 \addtocounter{minutes}{-\time}
\newtheorem{theorem}{Theorem}%[section]
\newtheorem{lemma}{Lemma}%[section]
\newtheorem{proposition}{Proposition}[section]

\usepackage{amsmath, amssymb, amscd, amsxtra, txfonts}

\title{Quasiconformal mappings and H\"older continuity}

\begin{document}

\author[Kalaj]{David Kalaj}
\address{University of Montenegro, Faculty of Natural Sciences and
Mathematics, Cetinjski put b.b. 81000 Podgorica, Montenegro}
\email{davidkalaj@gmail.com}
\author[Zllaticanin]{Arsen Zllaticanin}
\address{Department of Mathematics, University Luigj Gurakuqi, Shkodra, Albania}
\email{arsen\_zn@yahoo.fr}

%\keywords{Harmonic mappings, quasiconformal mappings, Poisson equation}
%\email{saksman@mappi.helsinki.fi}
%\date{}
% \subjclass {Primary 30C55, Secondary 31C05}
%\dedicatory{This paper is dedicated to our authors.}
%\keywords{Planar harmonic mappings, Quasiconformal mapping,
%  Convex domains, Rado-Kneser-Choquet theorem}

\maketitle

%\def\thefootnote{}
%\footnotetext{
%\texttt{\tiny File:~\jobname .tex,
%          printed: \number\year-\number\month-\number\day,
%          \thehours.\ifnum\theminutes<10{0}\fi\theminutes}
%}
\makeatletter\def\thefootnote{\@arabic\c@footnote}\makeatother

\begin{abstract}
We establish that every $K$-quasiconformal mapping $w$ of the unit ball $\IB$ onto a $C^2$-Jordan domain $\Omega$ is H\"older continuous with constant $\alpha= 2-\frac{n}{p}$, provided that its weak Laplacean $\Delta w$ is in $  L^p(\IB)$ for some $n/2<p<n$. In particular it is H\"older continuous for every $0<\alpha<1$ provided that $\Delta w\in  L^n(\IB)$.

\end{abstract}

\maketitle
%\tableofcontents

\section{Introduction}\label{intsec}

In this paper $\mathbb{B}^n$ denotes the unit ball in $\Bbb R^n$, $n\ge 2$ and $S^{n-1}$ denotes the
unit sphere. Also we will assume that $n>2$ (the case $n = 2$ has
been already treated in \cite{trans}). We will consider the vector
norm $|x|=({\sum_{i=1}^n x_i^2})^{1/2}$ and the  matrix norms
%$|A|_2:=(\mathrm{trace}\ AA^t)^{1/2}=(\sum_{i,j=1}^n
%a_{i,j}^2)^{1/2}$ and induced norm
$|A|=\sup\{|Ax|: |x|=1\}$.

A homeomorphism $u:\Omega \to \Omega'$ between two open subsets
$\Omega$ and $\Omega'$ of Euclid space $R^n$ will be called a $K$
($K\ge 1$) {\it quasi-conformal } or shortly a q.c mapping if

(i) $u$ is absolutely continuous function in almost every segment
parallel to some of the coordinate axes and there exist the partial
derivatives which are locally $L^n$ integrable functions on
$\Omega$. We will write $u \in ACL^n$ and

(ii) $u$ satisfies the condition $$|\nabla u(x)|^n/K\leq J_u(x)\leq
Kl(\nabla u(x))^n,$$ at almost everywhere  $x$ in $\Omega$ where
$$l(\nabla u(x)):=\inf\{|\nabla u(x)\zeta|:|\zeta|=1\}$$ and $J_u(x)$
is the Jacobian determinant of $u$ (see \cite{OS}).

Notice that, for a continuous mapping $u$ the condition (i) is
equivalent to the condition that $u$ belongs to the Sobolev space
$W^{1,n}_\mathrm{loc}(\Omega)$.

Let $P$ be Poisson kernel i.e. the function
$$P(x,\eta)=\frac{1-|x|^2}{|x-\eta|^n},$$ and let $G$ be the Green
function i.e. the function
\begin{equation}\label{green1}
G(x,y)=c_n\left\{
            \begin{array}{ll}
              \left(\frac{1}{|x-y|^{n-2}}-\frac{1}{(|\,x|y|-y/|y|\, |
)^{n-2}}\right), & \hbox{if $n\ge 3$;} \\
              \log\frac{|x-y|}{|1-x\bar y|}, & \hbox{if $n=2$ and $x,y\in \mathbb{C}\cong \mathbb{R}^2$.}
            \end{array}
          \right.
\end{equation} where
$c_n=\frac{1}{(n-2)\Omega_{n-1}}$, and $\Omega_{n-1}$ is the measure
of $S^{n-1}$. Both $P$ and $G$ are harmonic for $|x|<1$, $x\neq y$ .

Let $f:S^{n-1}\to \Bbb R^n$ be a $L^p$, $p>1$  integrable function on the
unit sphere $S^{n-1}$ and let $g:\mathbb{B}^n\mapsto \Bbb R^n$ be continuous.
The weak solution of the equation (in the sense of distributions) $\Delta
u=g$ in the unit ball satisfying the boundary condition
$u|_{S^{n-1}}=f\in L^1(S^{n-1})$ is given by
\begin{equation}\label{e:POISSON}u(x)=P[f](x)- G[g](x): =\int_{S^{n-1}}P(x,\eta)f(\eta)d\sigma(\eta)-
\int_{B^{n}}G(x,y)g(y)dy,\, \end{equation} $|x|<1$. Here $d\sigma$
is Lebesgue $n-1$ dimensional measure of Euclid sphere satisfying
the condition: $P[1](x)\equiv 1$. It is well known that if $f$ and
$g$ are continuous in $S^{n-1}$ and in $\overline{B^n}$
respectively, then the mapping $u=P[f]-G[g]$ has a
continuous extension $\tilde u$ to the boundary and $\tilde u = f$ on $S^{n-1}$. % see \cite{ABR}.
If $g\in L^\infty$ then $G[g]\in C^{1,\alpha}(\overline{B^n})$. See
\cite[Theorem~8.33]{gt} for this argument.

We will consider those solutions of the PDE $\Delta u = g$ that are
quasiconformal as well and investigate their Lipschitz character.

A mapping $f$ of a set $A$ in Euclidean $n$-space $\bold R^n$ into
$\bold R^n$, $n\geq 2$, is said to belong to the H\"older class
${\rm Lip}_\alpha(A)$, $\alpha>0$, if there exists a constant $M>0$
such that \begin{equation}\label{111}\vert f(x)-f(y)\vert \leq
M\vert x-y\vert^\alpha \end{equation} for all $x$ and $y$ in $A$. If
$D$ is a bounded domain in $\bold R^n$ and if $f$ is quasiconformal
in $D$ with $f(D)\subset\bold R^n$, then $f$ is in ${\rm
Lip}_\alpha(A)$ for each compact $A\subset D$, where
$\alpha=K_I(f)^{1/(1-n)}$ and $K_I(f)$ is the inner dilatation of
$f$. Simple examples show that $f$ need not be in ${\rm
Lip}_\alpha(D)$ even when $f$ is continuous in $\overline D$.

However O. Martio and R. N\"akki in \cite{mar1} showed that if $f$
induces a boundary mapping which belongs to ${\rm
Lip}_\alpha(\partial D)$, then $f$ is in ${\rm Lip}_\beta(D)$, where
$$\beta=\min(\alpha,K_I(f)^{1/(1-n)});$$ the exponent $\beta$ is
sharp.

In a recent paper of the second author and Saksman \cite{ks} it is proved the following result, if $f$ is quasiconformal mapping of the unit disk onto a Jordan domain with $C^2$ boundary such that its weak Laplacean $\Delta f\in L^p(\IB^2)$, for $p>2$, then $f$ is Lipschitz continous. The condition $p>2$ is necessary also. Further in the same paper they proved that if $p=1$, then $f$ is absolutely continuous on the boundary of $\partial \IB^2$. The results from \cite{ks} optimise in certain sense the results of the first author, Mateljevi\'c, Pavlovi\'c, Partyka, Sakan, Manojlovi\'c, Astala (\cite{JAM, pacific1, kalmat1, kalmat2, pk1, pk2, pk3, MP, pisa, mana, kave}), since it does not assume that the mapping is harmonic, neither its weak Laplacean is bounded.

We are interested in the condition under which the quasiconformal
mapping is in ${\rm Lip}_\alpha(B^n)$, for every $\alpha<1$. It follows form our results that
the condition that $u$ is quasiconformal and $|\Delta u|\in L^{p}$, such that $p>n/2$ guaranty that the selfmapping of the unit ball is in
 ${\rm Lip}_\alpha(B^n)$, where $\alpha= 2-\frac{p}{n}$. In particular if $p=n$, then $f\in {\rm Lip}_\alpha(B^n)$ for $\alpha<1$.

Our  result in several-dimensional case is  the following:
\begin{theorem}\label{th:firstn} Let $n\ge 2$ and let   $p>n/2$ and assume that $g\in L^p(\IB^n)$.
Assume that  $w$ is a $K$-quasiconformal  solution of  $\Delta w= g,$  that maps the unit ball onto a bounded Jordan domain $\Omega\subset\mathbb{R}^n$ with $C^{2}$-boundary.
\begin{itemize}
  \item If $p<n$, then $w$ is  H\"older  continuous with the H\"older constant $\alpha=2-\frac{n}{p}$.
  \item If $p=n$, then $w$ is H\"older  continuous for every $\alpha\in(0,1)$.
  \item If $n>p$ then $w$ is Lipschitz continuous.
\end{itemize}

\end{theorem}

\noindent The proof is given in the next section.

\section{ Proofs of the results}

In what follows, we say that a bounded Jordan domain $\Omega\subset\mathbb{R}^n$  has $C^2$-boundary if it is the image of the unit disc $\IB^n$ under a $C^2$-diffeomorphism of the whole complex plane onto itself. For planar Jordan domains this is well-known to be equivalent to the more standard definition, that requires the boundary to be locally isometric to the graph of a $C^2$-function on $\mathbb{R}^{n-1}$. In what follows, $\Delta $ refers to the distributional Laplacian. We shall make use of the following well-known facts.

\begin{proposition}[Morrey's inequality]
Assume that $n<p\le \infty$ and assume that $U$ is a domain in $\mathbf{R}^n$ with $C^1$ boundary. Then there exists a constant $C$ depending only on $n$, $p$ and $U$ so that \begin{equation}\label{morrey}
\|u\|_{C^{0,\alpha}(U)}\le C\|u\|_{W^{1,p}(U)}
\end{equation}
for every $u\in C^1(U)\cap L^p(U)$, where $$\alpha=1-\frac{n}{p}.$$
\end{proposition}

\begin{lemma} See e.g.\cite{kave}. \label{sobo}
Suppose that $w \in W^{2,1}_{loc}(\IB^n) \cap C( \overline{\IB^n}\, )$, that $h \in L^p(\IB^n)$ for some $1 < p <\infty $ and that
$$ \Delta w = h \; \mbox{ in }  \IB^n, \mbox{ with } w\big|_{ \mathbb{S}^{n-1}} = 0,
$$

a) If $1 < p < n$, then
$$ \| \nabla w \|_{L^q(\IB^n)} \leq c(p,n) \| h \|_{L^p(\IB^n)}, \qquad {q} = \frac{pn}{n-p}.
$$

b) If $p=n$ and $1<q<\infty$  then
$$  \| \nabla w \|_{L^{q}(\IB^n)}  \leq c(q,n) \| h  \|_{L^n(\IB^n)}.
$$

c) if $p>n$, then $$  \| \nabla w \|_{L^{\infty}(\IB^n)}  \leq c(p,n) \| h  \|_{L^n(\IB^n)}.
$$
\end{lemma}
Now we prove
\begin{lemma}\label{daki}
If $\Delta  u= g \in L^p$ and $r<1$, then $Du\in L^q(r\IB)$ for $q\le \frac{np}{n-p}$.
\end{lemma}

\begin{proof}[Proof of Lemma~\ref{daki}]
By writing $u=v+w$ from \eqref{e:POISSON},  and differentiating it  we have
\begin{equation}\label{e:POISSON3}Du(x) =Dv+Dw=\int_{S^{n-1}}\nabla P(x,\eta)f(\eta)d\sigma(\eta)-
\int_{\IB}\nabla_x G(x,y)g(y)dy. \end{equation}

Then $$\int_{r\IB} |Du(x)|^q dx=\int_{r\IB}\left|\int_{S^{n-1}}\nabla_x P(x,\eta)f(\eta)d\sigma(\eta)-
\int_{\IB}\nabla_x G(x,y)g(y)dy\right|^q dx.$$
Thus \[\begin{split}\|Du\|_{L^q(r\IB)} &=\|Dv\|_{{L^q(r\IB)}}+\|Dw\|_{L^q(r\IB)}
\\&\le \left(\int_{r\IB}\left|\int_{S^{n-1}}\nabla_x P(x,\eta)f(\eta)d\sigma(\eta)\right|^{1/q}\right)^{1/q}\\&+
\left(\int_{r\IB}\left|\int_{\IB}\nabla_x G(x,y)g(y)dy\right|^q dx\right)^{1/q}.\end{split}\]
There is a constant $C$ so that \begin{equation}\label{poi}|\nabla_x P(x,\eta)|\le \frac{C}{(1-|x|)^{n+1}}.\end{equation} From Lemma~\ref{sobo} and \eqref{poi}  we have $\|Du\|_{L^q(r\IB)}<\infty$.
\end{proof}

Now we formulate the following fundamental result of Gehring
\begin{proposition}\cite{gehring}\label{ger}
Let $f$ be a quasiconformal mapping of the unit ball $\IB^n$ onto a Jordan domain $\Omega$ with $C^2$ boundary. Then there is a constant $p=p(K,n)>n$ so that
$$\int_{\IB^n}|Df|^p<C(n, K, f(0), \Omega).$$
\end{proposition}

Then we prove

\begin{lemma}\label{arsen}
If $H:{\mathbf{R}}^{n} \rightarrow \mathbf{R}$ and $w=(w_{1},\ldots,w_{n}):A \rightarrow B$ (where $A, B$ are open subsets in ${\mathbf{R}}^{n}$) are functions from $C^{2}$ class, then:
$$
\Delta (H \circ w) = \sum_{i=1}^{n}  \pd[2]{H}{w_{i}^{2}}|\nabla{w_{i}}|^{2}+2\sum_{1\leq i < j\leq n} \pd[2]{H}{w_{i}\partial{w_{j}}}\left<\nabla{w_{i}},\nabla{w_{j}}\right>+\sum_{i=1}^{n} \pd{H}{w_{i}}\Delta{w_{i}}$$
\end{lemma}
\begin{proof}
For every $k\in(1,\ldots,n)$ we have: $$\pd{(H\circ w)(x_{1},\ldots,x_{n})}{x_{k}}=\sum_{i=1}^{n}\pd{H}{w_{i}}\pd{w_{i}}{x_{k}}.$$
Thus
\[
\begin{split}
\pd[2]{(H\circ{w})(x_{1},\ldots,x_{n})}{x_{k}^{2}}&=\sum_{i=1}^{n}\pd{[\pd{H}{w_{i}}\pd{w_{i}}{x_{k}}]}{x_{k}}
\\&=\sum_{i=1}^{n}\left[\pd{[\pd{H}{w_{i}}]}{x_{k}}\pd{w_{i}}{x_{k}}+\pd{H}{w_{i}}\pd[2]{w_{i}}{x_{k}^{2}}\right]
\\&=\sum_{i=1}^{n}\left[[\sum_{j=1}^{n}\pd[2]{H}{w_{j}\partial{w_{i}}}\pd{w_{j}}{x_{k}}]\pd{w_{i}}{x_{k}}\right]+\sum_{i=1}^{n}\pd{H}{w_{i}}\pd[2]{w_{i}}{x_{k}^{2}}
\\&
 =\sum_{i,j=1}^{n}\pd[2]{H}{w_{i}\partial{w_{j}}}[\pd{w_{i}}{x_{k}}\pd{w_{j}}{x_{k}}]+\sum_{i=1}^{n}\pd{H}{w_{i}}\pd[2]{w_{i}}{x_{k}^{2}}
\end{split}
\]
Now we have :

\[\begin{split}
\Delta(H\circ{w})&=\sum_{k=1}^{n}\pd[2]{(H\circ{w})(x_{1},\ldots,x_{n})}{x_{k}^{2}}\\&= \sum_{k=1}^{n}\left[\sum_{i,j=1}^{n}\pd[2]{H}{w_{i}\partial{w_{j}}}[\pd{w_{i}}{x_{k}}\pd{w_{j}}{x_{k}}]+\sum_{i=1}^{n}\pd{H}{w_{i}}\pd[2]{w_{i}}{x_{k}^{2}}\right]
\\&
=\sum_{i,j=1}^{n}\pd[2]{H}{w_{i}\partial{w_{j}}}[\sum_{k=1}^{n}\pd{w_{i}}{x_{k}}\pd{w_{j}}{x_{k}}]+\sum_{i=1}^{n}\pd{H}{w_{i}}[\sum_{k=1}^{n}\pd[2]{w_{i}}{x_{k}^{2}}]
\\&= \sum_{i=1}^{n}  \pd[2]{H}{w_{i}^{2}}|\nabla{w_{i}}|^{2}+2\sum_{1\leq i < j\leq n}\pd[2]{H}{w_{i}\partial{w_{j}}}\left<\nabla{w_{i}},\nabla{w_{j}}\right>+\sum_{i=1}^{n} \pd{H}{w_{i}}\Delta{w_{i}}
\end{split}
\]

\end{proof}
\begin{proof}[Proof of Theorem~\ref{th:firstn}]  It turns out that the approach of \cite{pisa}, where  the use of distance functions was initiated, is flexible enough for further development.

\emph{In the sequel we say $a \approx b$ if there is a constant $C\ge 1$ such that $a/C \le b \le C a$; and we say $a\lesssim b$ if there is a constant $C>0$ such that $a \le C b $.}

By our assumption on the domain, we may fix a diffeomorphism
$\psi:\overline{\Omega}\to\overline{\IB^n}$ that is $C^2$ up to the boundary.
 Denote $H:=1-|\psi|^2$, whence $H$ is  $C^2$-smooth in $\overline{\Omega}$  and vanishes on $\partial \Omega$ with $|\nabla H|\approx 1$ in a neighborhood of $\partial\Omega.$ We may then  define $h:\IB^n\to[0,1] $ by setting
\[
h(z):=H\circ w(z)= 1-|\psi(w(z))|^2\quad\textrm{for }\; z\in \IB^n.
\]
The quasiconformality of $f$ and the behavior of $\nabla H$ near $\partial\Omega$ imply that there is $r_0\in (0,1)$ so that the weak gradients satisfy
\begin{equation}\label{grad_equiv0}
|\nabla h(x)|\approx |\nabla  w(x)| \quad\textrm{for }\; r_0\leq |x|<1.
\end{equation}
Moreover, by Lemma~\ref{daki}, for $q\in (1,\frac{np}{n-p}]$ , we have
$$
\|\nabla h(x)\|_{L^q(r_0 \IB^n)}\lesssim \|\nabla  w(x)\|_{L^q(r_0 \IB^n)}\leq C.
$$
It follows that for any $q\in (1,\frac{np}{n-p}]$ we have that
\begin{equation}\label{grad_equiv1}
\nabla h\in L^q(\IB^n) \quad\textrm {if and only if }\; \nabla w\in L^q(\IB^n) .
\end{equation}

A direct computation (from Lemma~\ref{arsen}) by using  the fact that $H\in C^2$ is real valued, we obtain
\begin{equation}\label{laplacen}
|\Delta h|\lesssim   |\nabla w|^2+|g|.
\end{equation}

The higher integrability of quasiconformal self-maps of $\IB^n$ makes sure that $\nabla (\psi\circ w)\in L^q(\IB^n)$ for some $q>n$, which implies that  $ \nabla w\in L^q(\IB^n)$. By combining this with the fact that $g\in L^p(\IB^n)$ with $p>1,$ we deduce that $\Delta h\in L^{r}(\IB^n)$ with $r=\textrm{min}(p,q/2)>1.$
We use  bootstrapping argument based on the following observation: in our situation
\begin{equation}\label{boothstrap}
\textrm{if }\; \nabla w\in L^q(\IB^n)\;\; \textrm{with }\; n<q<2n,\quad \textrm{then  }\;  \nabla w\in L^{na/(2n-a)}(\IB^n),
\end{equation} where $a=q\wedge 2p$.
In order to prove \eqref{boothstrap}, assume that $\nabla w\in L^q(\IB^n)$ for an exponent $q\in (n,2n).$ Then \eqref{laplacen} and our assumption on $g$ verify that $\Delta h\in L^{q/2\wedge p}(\IB^n).$ Since $h$ vanishes continuously on the boundary $\partial \IB^n$, we may apply Lemma~\ref{sobo}(a) to obtain that $\nabla h\in L^{na/(2n-a)}(\IB^n)$ which yields the claim according to \eqref{grad_equiv1}.

We then claim that in our situation one has   $ \nabla w\in L^q(\IB^n)$ with some exponent $q>2n$. To prove that, fix an exponent $q_0>n$ obtained from the higher integrability of the quasiconformal map $w$ so that  $ \nabla w\in L^{q_0}(\IB^n)$. By diminishing $q_0$ if needed, we may well assume that
$q_0\in (n,2n)$ and
$
q_0\not\in\{{2^m}/{(2^{m-1}-1)},\; m=3,4,\ldots\}.
$
Then we may iterate \eqref{boothstrap} and deduce inductively that  $\nabla w\in L^{q_k}(\IB^n)$ for $k=0,1,2\ldots k_0$, where
the indexes $q_k$ satisfy the recursion $q_{k+1}=\frac{nq_k}{2n-q_k}$ and $k_0$ is the first index such that $q_{k_0}>2n$. Such an index exists since by induction we have the relation $(1-n/q_{k})=2^k (1 - n/q_0),$  for $k\geq 0.$ So $q_k>n$.  If $q_k\le 2n$, then we have $\limsup_{k\to\infty} (1-n/q_{k})=\infty$ which is impossible.

Thus we may assume that $ \nabla w\in L^q(\IB^n)$ with  $q>2n$. At this stage  \eqref{laplacen} shows that $\Delta h\in L^{p\wedge (q/2)}(\IB^n).$ As  $p\wedge (q/2)=p,$   Lemma \ref{sobo}(a) verifies that $\nabla h\in L^{np/(n-p)}(\IB^n)$. Finally,  by \eqref{grad_equiv1} we have the same conclusion for $\nabla w$, and hence by Morrey's inequality  $w$ is H\"older continuous with the constant $c=\alpha=2-\frac{n}{p}$ as claimed.
\end{proof}

If follows from the proof of the previous theorem that
\begin{theorem}\label{th:firstn2} Assume that  $g\in L^2(\IB^n)$.
If $w$ is a $K$-quasiconformal  solution of  $\Delta w= g,$  that maps the unit disk onto a bounded Jordan domain $\Omega\subset\mathbb{R}^n$ with $C^{2}$-boundary, then $Dw\in L^p(\IB^n)$ for every $p<\infty$.
\end{theorem}

\end{document}